\documentclass[fullpage]{article}

\usepackage[french,english]{babel}

\usepackage[T1]{fontenc}
\usepackage[utf8]{inputenc}

\usepackage{soul} %pour souligner \ul avec respect des retours à la ligne
\usepackage{enumerate} %pour les \begin{enumerate}[(i)] par exemple

\usepackage{amsfonts,amssymb,amsmath,amsthm}
\usepackage{bbm}
\usepackage{pgf}
\usepackage{xcolor}
\usepackage{hyperref}
\usepackage{enumitem}

\hypersetup{colorlinks=true,citecolor={red}}

\numberwithin{equation}{section}
\def\PP{\mathbb{P}}

\def\RR{\mathbb{R}}

\def\EE{\mathbb{E}}
\def\11{\mathbbm{1}}

\def\E{\mathbb{E}}
\def\P{\mathbb{P}}
\def\R{\mathbb{R}}

\def\N{\mathbb{N}}

\def\d{\partial}

\def\cM{{\cal M}}

\newtheorem{thm}{Theorem}[section]
\newtheorem{lem}[thm]{Lemma}

\newtheorem{prop}[thm]{Proposition}

\theoremstyle{remark}
\newtheorem{rem}{Remark}

\newcommand{\vertiii}[1]{{\left\vert\kern-0.25ex\left\vert\kern-0.25ex\left\vert #1 
    \right\vert\kern-0.25ex\right\vert\kern-0.25ex\right\vert}}
\begin{document}

\title{Stochastic approximation of quasi-stationary distributions for diffusion processes in a bounded domain}

\author{Michel Bena\"im$^1$, Nicolas Champagnat$^{2}$, Denis Villemonais$^{2}$}

\footnotetext[1]{Universit\'e de Neuch\^atel, Switzerland}
\footnotetext[2]{Universit\'e de Lorraine, CNRS, Inria, IECL, F-54000 Nancy, France \\
  E-mail: Nicolas.Champagnat@inria.fr, Denis.Villemonais@univ-lorraine.fr}

\maketitle

\begin{abstract}
  We study a random process with reinforcement, which evolves
  following the dynamics of a given diffusion process in a bounded
  domain and is resampled according to its occupation measure when it
  reaches the boundary. We show that its occupation measure converges
  to the unique quasi-stationary distribution of the diffusion process
  absorbed at the boundary of the domain. Our proofs use recent
  results in the theory of quasi-stationary distributions and
  stochastic approximation techniques.
\end{abstract}

\selectlanguage{french} 

\begin{abstract}
	Nous étudions un processus stochastique avec renforcement, qui évolue suivant une diffusion dans un domaine borné, avec ré-échantillonnage  suivant sa mesure d'occupation lorsqu'il atteint la frontière. Nous montrons que sa mesure d'occupation converge vers l'unique distribution quasi-stationnaire de la diffusion absorbée au bord du domaine. Nos preuves s'appuient sur des résultats récents en théorie des distributions quasi-stationnaires  et sur des techniques d'approximation stochastique.
\end{abstract}

\selectlanguage{english} 

\noindent\textit{Keywords:} random processes with reinforcement, stochastic approximation, pseu\-do-asymptotic trajectories, quasi-stationary distributions.

\medskip\noindent\textit{2010 Mathematics Subject Classification.} Primary: 60B12, 60J60, 60B10, 60F99; Secondary: 60J70.

\section{Introduction}
\label{sec:intro}

Let $(\Omega,({\cal F}_t)_{t\in[0,+\infty)},(X_t)_{t\in [0,+\infty)},(\PP_x)_{x\in E\cup \{\partial\}})$ be a
time homogeneous Markov process with state space $E\cup\{\partial\}$, where $E$ is a measurable space and
 $\d\not\in E$ is an absorbing state for the process.  This means that $X_s=\partial$ implies $X_t=\partial$ for all
$t\geq s$, $\P_x$-almost surely for all $x\in E$ and, in particular,
$$
\tau_\partial:=\inf\{t\geq 0,X_t=\partial\}
$$
is a stopping time. We also assume that
$\PP_x(\tau_\partial<\infty)=1$ and $\PP_x(t<\tau_\partial)>0$ for all
$t\geq 0$ and $\forall x\in E$.

We consider a random process $(Y_t)_{t\geq 0}$ with reinforcement,
which evolves following the dynamic of $X$ when it lies in $E$ and
which is resampled according to its occupation measure when it reaches
$\d$. More precisely, given a probability measure $\mu$ on $E$, we set
$$
Y_t=\sum_{k=1}^\infty \mathbbm{1}_{t\in[\theta_{k-1},\theta_k)}X^{(k)}_{t-\theta_{k-1}},\quad\forall t\geq 0,
$$
where $\theta_0=0$, 
\begin{itemize}
\item $(X^{(1)}_t,t\geq 0)$ is a realization of the process $(X_t,t\geq 0)$ with $X^{(1)}_0\sim\mu$ (i.e.\ under $\PP_\mu$) and
  the stopping time $\theta_1$ is defined as $\theta_1=\tau^{(1)}_\d$ the first hitting time of $\d$ by $X^{(1)}$,
\item given $X^{(1)}$, $(X^{(2)}_t,t\geq 0)$ is a realization of the process $(X_t,t\geq 0)$ with $X^{(2)}_0\sim\mu_{\theta_1}$,
  where
  $$
  \mu_{\theta_1}=\frac{1}{\theta_1}\int_0^{\theta_1}\delta_{Y_s}\,ds
  $$
  and $\theta_2-\theta_1=\tau^{(2)}_\d$ the first hitting time of $\d$ by $X^{(2)}$,
\item for all $k\geq 1$, given $X^{(1)},X^{(2)},\ldots,X^{(k)}$ ,$(X^{(k+1)}_t,t\geq 0)$ is a realization of the process $(X_t,t\geq
  0)$ with $X^{(k+1)}_0\sim\mu_{\theta_k}$,
  where
  $$
  \mu_{\theta_k}=\frac{1}{\theta_k}\int_0^{\theta_k}\delta_{Y_s}\,ds
  $$
  and $\theta_{k+1}-\theta_k=\tau^{(k+1)}_\d$ the first hitting time of $\d$ by $X^{(k+1)}$.
\end{itemize}
We also define for all $t\geq 0$
$$
\mu_t=\frac{1}{t}\int_0^t\delta_{Y_s}\,ds,\quad\text{i.e.}\quad \mu_t(f)=\frac{1}{t}\int_0^t f(Y_s)\,ds,\quad\forall
f\in\mathcal{B}_b(E).
$$

%Processus absorb\'e abstrait, on le fait rena\^itre selon sa mesure d'occupation. Question: comment se comporte la mesure
%d'occupation.

This process has been studied in several situations, with the main
goal of proving an almost sure convergence result for the occupation
measure $\mu_t$ when $t\rightarrow+\infty$. In the finite state space
case and in a discrete time setting, Aldous, Flannery and
Palacios~\cite{AldousFlanneryEtAl1988} solved this problem by showing
that the proportion of colours in a P\'olya urn type process converges
almost surely to the left eigenfunction of the replacement matrix,
which was also identified as the quasi-stationary distribution of a
corresponding Markov chain (we refer the reader to the
surveys~\cite{MV12,vanDoornPollett2013} and to the
book~\cite{ColletMartinezSanMartin} for general references on
quasi-stationary distributions; basic facts and useful results on
quasi-stationary distributions are also reminded in
Section~\ref{sec:main-discrete-time}). Under a similar setting but
using stochastic approximation techniques, Bena\"im and
Cloez~\cite{BenaimCloez2015} and Blanchet, Glynn and
Zheng\cite{BlanchetGlynnEtAl2016} independently proved the almost sure
convergence of the occupation measure $\mu_t$ toward the
quasi-stationary distribution of $X$. These works have since been
generalized to the compact state space case by Bena\"im, Cloez and
Panloup~\cite{BenaimCloezEtAl2017} under general criteria for the
existence of a quasi-stationary distribution for $X$. Continuous time
diffusion processes with smooth bounded killing rate on compact Riemanian
manifolds have been recently considered by Wang, Roberts and
Steinsaltz~\cite{WangRobertsEtAl2020}, who show that a similar
algorithm with weights also converges toward the quasi-stationary
distribution of the underlying diffusion process. Recently, Mailler
and Villemonais~\cite{MaillerVillemonais2018} have proved such a
convergence result for processes with smooth and bounded killing rate
evolving in non-compact (more precisely unbounded) spaces using a
measure-valued P\'olya process representation of this reinforced
algorithm.

The aim of the present paper is to solve the question of convergence
of the occupation measure toward the quasi-stationary distribution of
$X$ when this process is a uniformly elliptic diffusion evolving in an
open bounded connected open set $D$ with $C^2$ boundary $\d D$, with
hard killing when the process hits the boundary. This answers
positively the open problem stated in Section~8
of~\cite{BenaimCloezEtAl2017}. Note that the difficulty is twofold:
firstly, the state space $E=D$ is an open domain in $\R^d$ and is thus
non-compact; secondly, the absorption occurs through killing at the
boundary, which corresponds to an infinite killing rate.

Our main assumptions concern the $C^2$ regularity of the domain and of the
parameters of the diffusion $X$. They are satisfied in particular if the coefficients of the
stochastic differential equation satisfied by $X$ are H\"older
continuous. Our assumptions ensure the existence of a unique
quasi-stationary distribution $\alpha$ for $X$ and allows us to prove
the almost sure convergence of the occupation measure
$(\mu_t)_{t\geq 0}$ toward $\alpha$.  %{\color{blue}
As in~\cite{Benaim1999,BenaimCloez2015,WangRobertsEtAl2020}, we  make use of stochastic approximation techniques (in the sense of~\cite{Benaim1999,BenaiemHirsch1996}). These works strongly rely on the proof techniques of~\cite{MetivierPriouret1987}, which are based on technical regularity results that do not adapt well to the present setting. Our proof uses instead 
recent advances in the theory of quasi-stationary distributions~\cite{ChampagnatVillemonais2016} together with coupling arguments, in
order to prove that the occupation measure dynamics are globally asymptotically stable. Combined with the general results on
asymptotic pseudo trajectories of~\cite{Benaim1999,BenaiemHirsch1996}, this entails the almost sure convergence of the occupation
measure. %}

The paper is organised as follows. In
Section~\ref{sec:self-interacting}, we state our main assumptions and
results. In Section~\ref{sec:main-discrete-time}, we gather useful
general results on quasi-sta\-tio\-na\-ry distributions
from~\cite{ChampagnatVillemonais2016,ChampagnatVillemonaisEtAl2017} and
prove new general results on a key operator $A$, which has its own
interest and should be useful for future adaptation of the methods
developed below. Section~\ref{sec:proof} is devoted to the proof of
our main result, which consists in checking that the occupation measure
of the resampling points is (up to a time change and linearization) an
asymptotic pseudo-trajectory of a measure-valued dynamical system related to the operator $A$ (we refer the reader to~\cite{Benaim1999} %and~\cite{MetivierPriouret1987}
for an introduction to asymptotic pseudo-trajectories and their use in
stochastic approximation theory).%\footnote{\color{red}Pour Michel : est-ce la bonne fa\c{c}on de citer~~\cite{MetivierPriouret1987}}.

\section{Main result}
\label{sec:self-interacting}

From now on, we consider a diffusion process $(X_t,t\geq 0)$ in a connected bounded open set $D$ of $\RR^d$, $d\geq 2$ with $C^2$
boundary $\d D$ and absorbed at $\d D$. We assume that $X$ is solution to the SDE
\begin{equation}
  \label{eq:SDE}
  dX_t=\sigma(X_t)dB_t+b(X_t)dt,
\end{equation}
where $(B_t,t\geq 0)$ is a
$r$-dimensional Brownian motion, $b:D\rightarrow\R^d$ is bounded and continuous and $\sigma:D\rightarrow \R^{d\times r}$ is
continuous, $\sigma\sigma^*$ is uniformly elliptic and for all $\rho>0$,
\begin{align}
  \sup_{x,y\in D,\ |x-y|=\rho}\frac{|\sigma(x)-\sigma(y)|^2}{\rho}\leq g(\rho)
  \label{eq:hyp-sigma-priola-wang} 
\end{align}
for some function $g$ such that $\int_0^1 g(r)dr<\infty$. Note that, in this case, the process $(Y_t,\mu_t)_{t\geq 0}$ described in the introduction is
well-defined since one can prove that $\theta_k\rightarrow+\infty$
a.s.~\cite[Lemma 8.1]{BenaimCloezEtAl2017}.

In~\cite[Section~5.3]{ChampagnatCoulibalyEtAl2018}, it was proved
that, under the above regularity assumptions, the killed diffusion
process $X$ admits a unique quasi-stationary distribution,
\textit{i.e.} a probability measure $\alpha$ on $D$ such that
\begin{align*}
  \alpha=\P_\alpha(X_t\in\cdot\mid t<\tau_\d),\ \forall t\geq 0,
\end{align*}
where $\tau_\d$ denotes the hitting time of $\d D$ by the
process. {Moreover, it is well known that, in this case, there exists a
positive constant $\lambda_0$ such that
$\P_\alpha(t<\tau_\d)=\exp(-\lambda_0 t)$ for all $t\geq 0$ (see
Section~\ref{sec:main-discrete-time} for more results on
quasi-stationary distributions).}

\begin{rem}
  \label{rem:expo-cv}
  In fact, the result
  of~\cite[Section~5.3]{ChampagnatCoulibalyEtAl2018} is stronger and
  entails the exponential convergence in total variation norm of the
  conditional law of $X$ toward $\alpha$, uniformly in the initial
  distribution. The proof relies on the fact that Conditions~(A1)
  and~(A2) as enunciated in the next section are satisfied by the
  process $X$ (see Section~\ref{sec:main-discrete-time} for details
  and additional properties).
\end{rem}

\begin{rem}
  \label{rem:manifold-and-1d}
  This last property was also proved to hold true for general
  one-di\-men\-sio\-nal diffusions in $D=[a,+\infty)$ or $D=[a,b]$ absorbed
  at the boundary of $D$ and coming down from infinity
  in~\cite{ChampagnatVillemonais2015} and for diffusion processes $X$
  in compact, connected $C^2$ manifolds $M$ with $C^2$ boundary $\d M$
  absorbed at $\d M$ when the infinitesimal generator of $X$ is given
  by $L=\frac{1}{2}\Delta+Z$, where $\Delta$ is the Laplace-Beltrami
  operator and $Z$ is a $C^1$ vector field
  in~\cite{ChampagnatCoulibalyEtAl2018}. All the results of this
  paper, and in particular the next one, can be extended to these two
  situations.
\end{rem}

The main result of this article is the following one.

\begin{thm}
  \label{thm:main-d>1}
  For all bounded measurable function $f:D\rightarrow \R$, one has 
  \[
  \mu_t f\xrightarrow[t\rightarrow+\infty]{}\alpha f\quad\text{a.s.}
  \]
 Moreover, $\theta_n/n\to 1/\lambda_0$ almost surely when $n\rightarrow+\infty$.
\end{thm}

% \begin{rem}
%   En fait, la convergence a lieu pour toute fonction test mesurable born\'ee.
% \end{rem}

\section{Properties of the {Green operator}}
\label{sec:main-discrete-time}

{The results of Subsection~\ref{subsec:gp} are valid for general
absorbed Markov processes, not only for diffusion processes absorbed
at the boundary of a domain. In Subsection~\ref{sec:dyn-syst}, we
provide properties on the measure-valued dynamical system induced by
the Green operator of the process. Although not specific to diffusion
processes, the later part uses the fact that the semi-group of the
underlying process is Lipschitz regular.

\subsection{General properties}}
\label{subsec:gp}

Let us consider in this section a Markov process $(X_t,t\geq 0)$ on a measurable space $E\cup\{\partial\}$, absorbed in $\d$ at
time
$$
\tau_\partial:=\inf\{t\geq 0,X_t=\partial\},
$$
assumed a.s.\ finite. We also assume that $\PP_x(t<\tau_\partial)>0$ for all $t\geq 0$ and all $x\in E$.

A probability measure $\alpha$ on $E$ is called a \textit{quasi-stationary distribution} if
$$
\PP_\alpha(X_t\in\cdot\mid t<\tau_\d)=\alpha,\quad\forall t\geq 0.
$$
% We refer the reader to~\cite{MV12,vanDoornPollett2013,ColletMartinezSanMartin} and references therein for extensive developments
% and several references on the subject. 
It is well known that a probability measure $\alpha$ is a quasi-stationary distribution if and only if there exists a probability measure $\mu$ on $E$ such that
\begin{align}
\label{eq:conv}
\lim_{t\rightarrow+\infty} \PP_\mu(X_t\in A\mid t<\tau_\d)=\alpha(A)
\end{align}
for all measurable subsets $A$ of $E$. The fact that $\alpha$ is a quasi-stationary distribution also implies the existence of a constant $\lambda_0>0$ such that 
\begin{align}
\label{eq:lambda_0}
\PP_\alpha(t<\tau_\d)=e^{-\lambda_0 t}.
\end{align}

In~\cite{ChampagnatVillemonais2016}, the authors provide a necessary and sufficient condition on $X$ for the existence of a
probability measure $\alpha$ on $E$ and constants $C,\gamma>0$ such that
\begin{equation}
  \label{eq:conv-exp}
  \left\|\PP_\mu(X_t\in\cdot\mid t<\tau_\d)-\alpha\right\|_{TV}\leq C e^{-\gamma t},\quad\forall \mu\in\mathcal{M}_1(E),\quad t\geq 0,
\end{equation}
where $\mathcal{M}_1(E)$ is the set of probability measures on $E$ and $\|\cdot\|_{TV}$ is the total variation norm defined as
$\|\mu_1-\mu_2\|_{TV}=\sup_{f\in\mathcal{B}_b(E),\,\|f\|_\infty\leq 1}|\mu_1(f)-\mu_2(f)|$ for all $\mu_1,\mu_2\in\mathcal{M}_1(E)$,
where $\mathcal{B}_b(E)$ is the set of bounded measurable functions on $E$. This immediately implies that $\alpha$ is the unique
quasi-stationary distribution of $X$ and that~\eqref{eq:conv} holds for any initial probability measure $\mu$.

The necessary and sufficient condition for~\eqref{eq:conv-exp} is given by the existence of a probability measure $\nu$ on $E$ and of
constants $t_0,c_1,c_2>0$ such that
\begin{equation}
  \PP_x(X_{t_0}\in \cdot\mid t_0<\tau_\d)\geq c_1\nu,\quad\forall x\in E  \tag{A1}\label{eq:A1}
\end{equation}
and
\begin{equation}
  \PP_\nu(t<\tau_\d)\geq c_2\PP_x(t<\tau_\d),\quad\forall t\geq 0,\ x\in E.  \tag{A2}\label{eq:A2}
\end{equation}

Under Conditions~(A1) and~(A2), it follows from the general results of~\cite[Prop.\,2.3]{ChampagnatVillemonais2016} that there exists a bounded
function $\eta:E\rightarrow[0,\infty)$ such that $\alpha(\eta)=1$ and, for all $x\in E$ and all $t\geq 0$,
\begin{align}
\label{eq:eta-unif-conv}
\left|e^{\lambda_0 t}\PP_x(t<\tau_\d)-\eta(x)\right|\leq C' e^{-\gamma t}.
\end{align}
In the case of diffusion processes, $\eta$ is a nonnegative solution to $L\eta=-\lambda_0\eta$ where $L$ is the infinitesimal
generator of the process $X$ in the set of bounded measurable functions equiped with the $L^\infty$ norm. The constant $\gamma$ is
the same as in~\eqref{eq:conv-exp}. In particular, there exists a constant $C''$ such that
\begin{align}
\label{eq:eta-borne-unif}
\PP_x(t<\tau_\d)\leq C'' e^{-\lambda_0 t},\quad\forall t\geq 0,\ \forall x\in E.
\end{align}
One can actually obtain a better bound combining Theorem~2.1 and Equation~(3.2) of~\cite{ChampagnatVillemonaisEtAl2017}: there exists a
time $t_1>0$ and a constant $D$ such that, for all $t\geq t_1$ and all $x\in E$,
\begin{align}
  \label{eq:eta-better-bound}
  \left|e^{\lambda_0 t}\PP_x(t<\tau_\d)-\eta(x)\right|\leq D\eta(x) e^{-\gamma t}.
\end{align}
We may---and will---assume without loss of generality that $De^{-\gamma t_1}\leq 1/2$.

We denote by $P_t$ the (nonconservative) semigroup of the Markov process $(X_t,t\geq 0)$, acting on the set $\mathcal{B}_b(E)$ of
bounded measurable functions on $E$ and defined for all such function $f$ by
$$
P_tf(x)=\EE_x[f(X_t)\mathbbm{1}_{t<\tau_\d}],\quad\forall x\in E.
$$
Note that we made here the slight abuse of notation that $f(\d)\cdot 0=0$. Because of~\eqref{eq:eta-borne-unif}, we can define the
{Green} operator $A$ on $\mathcal{B}_b(E)$ as
\begin{equation}
  \label{eq:A}
Af(x)=\EE_x\left[\int_0^{\tau_\d}f(X_s)\,ds\right]=\int_0^\infty P_sf(x)\,ds
\end{equation}
and this operator is bounded on $\mathcal{B}_b(E)$ equiped with the $L^\infty$ norm. Let $\mathcal{M}_1(E)$ be the set of probability
measures on $E$. For all $\mu\in\mathcal{M}_1(E)$, we also define the notation
$$
\mu Af=\int_E Af(x)\,\mu(dx)=\EE_\mu\left[\int_0^{\tau_\d}f(X_s)\,ds\right]=\int_0^\infty \mu P_sf\,ds,
$$
so that in particular $Af(x)=\delta_x Af$ and $\alpha A f=\int_0^\infty e^{-\lambda_0
  t}\alpha f\,dt=\alpha f/\lambda_0$. Since $A$ is bounded,
% Note that~\eqref{eq:eta-unif-conv} implies that the operator $A$ is bounded. In particular,
the operator $e^{tA}$ is well-defined for all $t\geq 0$.

\begin{prop}
  \label{prop:prop-A}
Assume that Conditions~(A1) and~(A2) are satisfied. Then, for all $\mu\in\mathcal{M}_1(E)$, all $f\in \mathcal{B}_b(E)$ and all $n\geq 1$, we have
  \begin{equation}
    \label{eq:prop-A-1}
    \left|\mu A^n f-\frac{\alpha(f)\mu(\eta)}{\lambda_0^n}\right|\leq \frac{(CC''+C')\|f\|_\infty}{(\lambda_0+\gamma)^n},    
  \end{equation}
  where the constants $C,C',C''$ and $\gamma$ are those involved in~\eqref{eq:conv-exp},~\eqref{eq:eta-unif-conv}
  and~\eqref{eq:eta-borne-unif}. We also have for some constant $B$
  \begin{equation}
    \label{eq:prop-A-2}
    \left\|\frac{\mu A^n}{\mu A^n\mathbbm{1}}-\alpha\right\|_{TV}\leq \frac{B}{\mu(\eta)}\left(\frac{\lambda_0}{\lambda_0+\gamma}\right)^n
  \end{equation}
  and for all $t\geq 0$,
  \begin{equation}
    \label{eq:prop-A-3}
    \left\|\frac{\mu e^{tA}}{\mu e^{tA}\mathbbm{1}}-\alpha\right\|_{TV}\leq \frac{B}{\mu(\eta)} e^{-t\,\frac{\gamma}{\lambda_0(\lambda_0+\gamma)}}.
  \end{equation}
\end{prop}

\begin{proof}
  We first check by induction that for all $n\geq 1$,
  \begin{equation}
    \label{eq:proof-prop-A-1}
    \mu A^n f=\int_0^\infty \frac{u^{n-1}}{(n-1)!}\mu P_u f\,du.
  \end{equation}
  This is of course true for $n=1$. Assuming it is true for a given $n\geq 1$, we have
  \begin{align*}
    \mu A^{n+1} f & =\int_0^\infty \mu P_s A^n f\,ds \\ & =\int_0^\infty \int_0^\infty \frac{t^{n-1}}{(n-1)!}\mu P_sP_t f\,dt\,ds
    \\ & =\int_0^\infty \mu P_u f \int_0^u \frac{t^{n-1}}{(n-1)!}\,dt\,du \\ & =\int_0^\infty \frac{u^{n}}{n!}\mu P_u f\,du,
  \end{align*}
  which concludes the induction. Then, it follows from~\eqref{eq:conv-exp},~\eqref{eq:eta-unif-conv} and~\eqref{eq:eta-borne-unif}
  that
  \begin{multline*}
    \left| \mu A^n f-\int_0^\infty \frac{u^{n-1}}{(n-1)!}\alpha(f) e^{-\lambda_0 u}\mu(\eta)\,du\right| \\
    \begin{aligned}
      & \leq \int_0^\infty \frac{u^{n-1}}{(n-1)!}\left|\mu P_u f-\alpha(f) e^{-\lambda_0 u}\mu(\eta)\right|\,du \\
      & \leq \int_0^\infty \frac{u^{n-1}}{(n-1)!}\left[\mu P_u\mathbbm{1}\left|\frac{\mu P_u f}{\mu
            P_u\mathbbm{1}}-\alpha(f)\right|+\alpha(f)\left|\mu
          P_u\mathbbm{1}-e^{-\lambda_0 u}\mu(\eta)\right|\right]\,du \\
      & \leq (CC''+C')\|f\|_\infty \int_0^\infty \frac{u^{n-1}}{(n-1)!}e^{-(\lambda_0+\gamma)u}\,du.
    \end{aligned}
  \end{multline*}
  The inequality~\eqref{eq:prop-A-1} follows.

  We then deduce from~\eqref{eq:prop-A-1} that
  \begin{align*}
    \left\|\frac{\mu A^n}{\mu A^n\mathbbm{1}}-\alpha\right\|_{TV} & \leq\frac{1}{\mu A^n\mathbbm{1}}\left[\left\|\mu
        A^n-\mu(\eta)\lambda_0^{-n}\alpha\right\|_{TV}+\left|\mu A^n\mathbbm{1}-\mu(\eta)\lambda_0^{-n}\right|\right] \\
    & \leq\frac{2(CC''+C')}{(\lambda_0+\gamma)^n\mu A^n\mathbbm{1}}.
  \end{align*}
  Now, it follows from~\eqref{eq:eta-better-bound} that
  \begin{align}
    \mu A^n \mathbbm{1} & \geq\int_{t_1}^\infty \frac{u^{n-1}}{(n-1)!}\mu P_u\mathbbm{1}\,du \notag \\
    & \geq\frac{\mu(\eta)}{2}\int_{t_1}^\infty \frac{u^{n-1}}{(n-1)!}e^{-\lambda_0 u}\,du \notag \\
    & =\frac{\mu(\eta)e^{-\lambda_0
        t_1}}{2}\left(\frac{t_1^{n-1}}{\lambda_0(n-1)!}+\frac{t_1^{n-2}}{\lambda_0^2(n-2)!}+\ldots+\frac{1}{\lambda_0^n}\right) \notag \\
    & \geq \frac{\mu(\eta)e^{-\lambda_0 t_1}}{2\lambda_0^n}. \label{eq:proof-prop-A-2}
  \end{align}
  Combining the last two inequalities entails~\eqref{eq:prop-A-2}.

  Similarly, for all $t\geq 0$, $f\in\mathcal{B}_b(E)$ and $\mu\in\mathcal{M}_1(E)$, we deduce from~\eqref{eq:prop-A-1} that
  \begin{multline*}
    \left| \frac{\mu e^{tA}f}{\mu e^{tA}\mathbbm{1}}-\alpha(f)\right| \\
    \begin{aligned}
      & \leq \frac{1}{\mu e^{tA}\mathbbm{1}}\sum_{n\geq 0}\frac{t^n}{n!}\left[\left|\mu A^n
          f-\mu(\eta)\lambda_0^{-n}\alpha(f)\right|+\alpha(f)\left|\mu(\eta)\lambda_0^{-n}-\mu A^n\mathbbm{1}\right|\right] \\
      & \leq \frac{2(CC''+C')\|f\|_\infty}{\mu e^{tA}\mathbbm{1}} e^{\frac{t}{\lambda_0+\gamma}}.
    \end{aligned}
  \end{multline*}
  Now, it follows from~\eqref{eq:proof-prop-A-2} that
  \begin{align*}
    \mu e^{tA}\mathbbm{1} & \geq \frac{\mu(\eta)e^{-\lambda_0 t_1}}{2}\,e^{\frac{t}{\lambda_0}}.
  \end{align*}
  The last two inequalities entail~\eqref{eq:prop-A-3}.
\end{proof}

\subsection{Properties of a measure-valued dynamical system}
\label{sec:dyn-syst}

We begin with the following proposition, which ensures that $A$ is regularizing.
% and, in particular, Using for instance the gradient estimates of Priola and Wang~\cite{PriolaWang2006}, one deduces that the
% operator $A$ is Feller in the sense that, for all $x\in D$, $\delta_x A$ is an operator from $C_b(D,\R)$ to itself, where
% $C_b(D,\R)$ is the set of bounded continuous functions from $D$ to $\R$.
In particular, for all $f\in C_b(D,\R)$ (which denotes the set of
bounded continuous functions from $D$ to $\R$),
$x\in D\mapsto \delta_x Af$ is in $C_b(D,\R)$. This Feller property implies that
$\nu\mapsto \nu A$ is continuous with respect to the weak topology on
the set $\cM(D)$ of non-negative measures with finite mass on
$D$. Similarly, one deduces that
$(t,\nu)\in[0,+\infty)\times \cM(D)\mapsto \nu e^{tA}\in\cM(D)$ is
continuous.

\begin{prop}
  \label{prop:propLip}
  For all bounded measurable functions $f:D\rightarrow\R$, the
  application $x\mapsto \delta_x Af$ is Lipschitz continuous, with
  Lipschitz norm proportional to $\|f\|_\infty$.
\end{prop}

\begin{proof}
  From Priola and Wang~\cite{PriolaWang2006}, one deduces that there
  exists a constant $C_{Lip}>0$ which does not depend on $f$ such that, for
  all $t>0$ and all $x,y\in D$,
  \[
    |\delta_x P_t f-\delta_y P_t f|\leq \frac{C_{Lip}}{1\wedge \sqrt{t}} \|f\|_\infty.
  \]
  Applying this inequality to $x\mapsto \delta_x P_t f$ at time $1$
  and using inequality~\eqref{eq:eta-borne-unif}, one deduces that
  \[
    |\delta_x P_{t+1} f - \delta_y P_{t+1} f|\leq C_{Lip} \|P_t f\|_\infty \leq C_{Lip} C''\|f\|_\infty e^{-\lambda_0 t}.
  \]
  As a consequence,
  \[
    |\delta_x Af-\delta_y Af|\leq |x-y|\int_0^1 \frac{C_{Lip}}{\sqrt{t}} \|f\|_\infty\,dt+ |x-y|\int_1^\infty  C_{Lip} C''\|f\|_\infty e^{-\lambda_0 (t-1)}\,dt,
  \]
  which concludes the proof of Proposition~\ref{prop:propLip}.
\end{proof}

The following proposition states the uniqueness of the
evolution equation satisfied by the continuous process
$(\nu e^{tA}/\nu e^{tA}\11_D)_{t\geq 0}$.

\begin{prop}
  \label{prop:propTheFlow}
  For each probability measure $\nu$ on $D$, the equation
\begin{align}
  \label{eq:eqTheFlow}
  \frac{d\varphi_t}{dt}=F(\varphi_t),\quad\varphi_0=\nu,
\end{align}
where $F$ is a measure valued function defined, for all non-negative
finite measures $\nu$ on $D$ by
\begin{equation}
\label{eq:def-F}
F(\nu)=\nu A-(\nu A\11_D)\,\nu,
\end{equation}
admits a unique weak solution in $C([0,+\infty),\cM(D))$, where $\cM(D)$ is equiped with the weak topology, in the sense that,
for all bounded continuous function $f:D\rightarrow\mathbb{R}$ and all $t\geq 0$,
\[
\varphi_t(f)=\varphi_0(f)+\int_0^tF(\varphi_s)(f)\,ds.
\]
In addition, this unique weak solution takes its values in $\mathcal{M}_1(D)$ and is given by $\varphi_t=\nu e^{tA}/\nu e^{tA}\11_D$.
\end{prop}

\begin{proof}
  The fact that $(\nu e^{tA}/\nu e^{tA}\11_D)_{t\geq 0}$
  satisfies~\eqref{eq:eqTheFlow} is immediate. Let us check that this equation has no other solution. In order to do so, we
  consider one of its solutions $(\varphi_t)_{t\geq 0}$ and introduce
  the measure valued process defined by
  \[
    \tilde{\varphi}_t=\exp\left(\int_0^t \varphi_s(A\11_D)\,ds\right)\varphi_t,\quad \forall t\geq 0.
  \]
  This process is weak solution to the linear evolution equation
  \[
    \frac{\d \tilde{\varphi}_t}{\d t}=\tilde\varphi_t A,\ \tilde\varphi_0=\nu,
  \]
  whose unique weak solution is $t\mapsto \nu e^{tA}$. {Indeed, let
  $t \mapsto \mu_t, \nu_t$ be two weak solutions to the linear
  equation. Set $|\mu_t - \nu_t| = \sup_f |\mu_t f - \nu_t f|$ where
  the supremum is taken over the set of continuous function
  $f : D \mapsto \RR$ such that $\|f\|_\infty \leq 1$. Then
  $t \mapsto |\mu_t - \nu_t|$ is lower semicontinuous, hence
  measurable, as a supremum of continuous functions. Thus, by
  Gronwall's lemma (measurable version, see~\cite{ethier-kurtz-86}),
  $|\mu_t - \nu_t| \leq |\mu_0 - \nu_0| e^{\|A\| t}.$ This proves
  uniqueness.}

  As a
  consequence, for all $t\geq 0$,
  \[
    \varphi_t=\frac{\tilde{\varphi}_t}{\tilde{\varphi}_t\11_D}= \frac{\nu e^{tA}}{\nu e^{tA}\11_D},
  \]
  which concludes the proof of Proposition~\ref{prop:propTheFlow}.
\end{proof}

\section{Proof of Theorem~\ref{thm:main-d>1}}
\label{sec:proof}

The general idea of the proof is inspired from~\cite{Benaim1999} and consists in proving that a time-change of the sequence of probability measures
\begin{equation}
\label{def:eta-n}
  \eta_n=\frac{1}{n}\sum_{i=1}^n \delta_{Z_i},\quad\text{where}\quad Z_i:=Y_{\theta_i}
\end{equation}
is an asymptotic pseudo-trajectory (see~\cite{Benaim1999} for the definition of an asymptotic pseudo-trajectory) of a measure-valued
dynamical system related to the normalized semigroup $\frac{\nu e^{tA}}{\nu e^{tA}\11_D}$. The asymptotic properties given in
Proposition~\ref{prop:prop-A} then allow to deduce that $\eta_n$ almost surely converges to $\alpha$. The proof is divided in three
steps. First, % we study in Subsection~\ref{sec:dyn-syst} the
              % dynamical system related to $\frac{\nu e^{tA}}{\nu
              % e^{tA}\11_D}$. Then,
we prove in Subsection~\ref{sec:SD} tightness properties on the
measure-valued process $(\mu_t)_{t\geq 0}$. The convergence of
$\eta_n$ to $\alpha$ is proved in Subsection~\ref{sec:eta-n}, using a
key lemma on asymptotic pseudo-trajectories properties for $\eta_n$,
proved in
Subsection~\ref{sec:pf-lemma-APT}. Theorem~\ref{thm:main-d>1} is then
be deduced from the convergence of $\eta_n$ using martingale arguments
in Subsection~\ref{sec:end-pf}.

\subsection{Tightness}
\label{sec:SD}

The following proposition entails that the paths of $(\mu_t,t\geq 0)$ are a.s.\ relatively compact in the set of probability measures on $D$ endowed with the weak topology.

\begin{prop}
  \label{prop:tightness1}
  For all $\varepsilon>0$, there exists $\eta>0$ such that, almost surely,
  \begin{align*}
   \liminf_{t\rightarrow+\infty}  \mu_{t}(\{x\in D:d(x,\d
    D)<\eta\}) \leq \varepsilon.
  \end{align*}
\end{prop}

\begin{proof}
  Let $\phi_D:D\rightarrow\RR_+$ be the distance to $\d D$. There exists a neighborhood $\mathcal{N}$ of $\d D$ in $D$ where $\phi_D$
  is $C^2_b$ so that we can apply It\^o's formula: for all $t\geq 0$ such that $Y_t\in\mathcal{N}$,
  \[
  d\phi_D(Y_t)=\left(\sigma(Y_t)^*\nabla \phi_D(Y_t)\right)\cdot dB_t+\nabla \phi_D(Y_t)\cdot
  b(Y_t)dt+\frac{1}{2}\text{Tr}(\sigma(Y_t)^*D^2\phi_D(Y_t)\sigma(Y_t))dt.
  \]
  We introduce the random time-change $\tau(t)$ such that
  \[
  \int_0^{\tau(t)}\left(\11_{Y_s\in\mathcal{N}}\left\|\sigma(Y_s)^*\nabla \phi_D(Y_s)\right\|^2_2
    +\11_{Y_s\not\in\mathcal{N}}\right)ds=t
  \]
  and we observe that there exist constants $0<c_0<C_0<\infty$ such that $c_0\leq\tau'(t)\leq C_0$ for all $t\geq 0$. Then, there exists a
  Brownian motion $W$ such that the process $Z_t:=\phi_D(Y_{\tau(t)})$ satisfies
  \[
  d Z_t=dW_t+H_t dt,\quad\forall t\text{ s.t.\ }Z_t\in\mathcal{N},
  \]
  where the process $H$ is progressively measurable and bounded by a constant $\bar{H}>0$.

  We introduce $a>0$ such that $\{x\in D:d(x,\d D)\leq 2a\}\subset\mathcal{N}$ and the reflected drifted Brownian motion $(\bar{Z}_t,t\geq
  0)$ solution to
  \[
  d\bar{Z}_t=d W_t-\bar{H}dt+dL^0_t-dL^a_t,\quad\forall t\geq 0
  \]
  and such that $\bar{Z}_0=\phi_D(Y_0)\wedge a$, where $L^x_t$ is the local time of $\bar{Z}$ at $x$ at time $t$.

  Since the jumps of $Z$ are positive, one can prove following~\cite[Prop. 2.2]{Villemonais2011} that $\bar{Z}_t\leq Z_t$
  a.s.\ for all $t\geq 0$. Moreover, the process $\bar{Z}$ is ergodic and satisfies almost surely
  \[
  \frac{1}{t}\int_0^t \delta_{\bar{Z}_s}ds\xrightarrow[t\rightarrow+\infty]{} m,
  \]
  where $m(dx)=C e^{-2\bar{H} x}\mathbbm{1}_{[0,a]}(x)\,dx$ is the stationary distribution of $\bar{Z}$ on $[0,a]$.

  Now, for all $\varepsilon>0$, there exists $\eta>0$ such that
  $m(0,\eta)<\varepsilon$. Hence, almost surely for all $t$ large
  enough
  \begin{align*}
    \mu_{\tau(t)}(\{x\in D:d(x,\d
    D)<\eta\}) &
    \leq\frac{1}{\tau(t)}\int_0^{\tau(t)}\11_{\bar{Z}_{\tau^{-1}(s)}<\eta}ds\leq\frac{1}{\tau(t)}\int_0^{t}\11_{\bar{Z}_u<\eta}\tau'(u)du \\ & \leq
    \frac{C_0}{c_0 t}\int_0^t \11_{\bar{Z}_u<\eta}du\leq\frac{C_0\varepsilon}{c_0}.
  \end{align*}
  Since $\tau:\R_+\rightarrow\R_+$ is continuous and $\tau'(t)\geq c_0$
  for all $t\geq 0$, this concludes the proof of
  Proposition~\ref{prop:tightness1}.
\end{proof}

The previous proposition entails that, for all $\varepsilon>0$, there
exists $\eta>0$ such that, almost surely,
$\theta_n\mu_{\theta_n}(\{x\in D:d(x,\d D)\geq \eta\})\geq
(1-2\varepsilon) \theta_n$ for $n$ large enough. The following
proposition is of a slightly different nature and it will be used
later in order to prove that there exists a constant $c>0$ such that,
almost surely,
$\theta_n\mu_{\theta_n}(\{x\in D:d(x,\d D)\geq \eta\})\geq cn$ for $n$
large enough.
\begin{prop}
  \label{prop:tightness2} For all $\varepsilon>0$, there exists $\eta>0$ such that, almost surely, one has
  \begin{equation}
    \label{eq:prop-tight-2a}
    \liminf_{n\rightarrow+\infty} \frac{1}{n}\sum_{i=1}^n \11_{d(Y_{\theta_i},\d D)\geq \eta}\geq 1-2\varepsilon
  \end{equation}
  and
  \begin{equation}
    \label{eq:prop-tight-2b}
    \liminf_{n\rightarrow+\infty} \frac{\theta_n}{n}\geq (1-2\varepsilon)c_0\mathbb{E}(\bar{T}_0),
  \end{equation}
  where $\bar{T}_0=\inf\{t\geq 0,\bar{Z}_t=0\}$, $\bar{Z}_0=\eta$ and the constant $c_0$ and the process $\bar{Z}$ were defined
  in the proof of Proposition~\ref{prop:tightness1}.
\end{prop}

\begin{proof}
  Fix $\varepsilon>0$. From Proposition~\ref{prop:tightness1}, there
  exists $\eta>0$ such that, almost surely,
  $\mu_{\theta_n}(\{x\in D:d(x,\d D)\geq \eta\})\geq
  1-2\varepsilon$ for $n$ large enough.  For all
  $k\in\N=\{1,2,\ldots\}$, we define the random  variable in $\mathbb{N}\cup\{+\infty\}$
  \[
    \upsilon_k=\inf\Big\{n\geq k,\,\mu_{\theta_n}(\{x\in D:d(x,\d D)\geq \eta\})<1-2\varepsilon\Big\},
  \]
  so that $\P(\cup_{k=1}^\infty \{\upsilon_k=+\infty\})=1$.
  We also define the sequence of points $(Z^k_n)_{n\geq 0}$ in $D$ by
  \[
    Z^k_n=\begin{cases}
      Y_{\theta_n}&\text{ if }n<\upsilon_k\\
      x_0&\text{ if }n\geq \upsilon_k,
      \end{cases}
  \]
  where $x_0$ is an arbirary point in $\{x\in D:d(x,\d D)\geq
  \eta\}$. % {\color{blue}
By definition, the law of $Y_{\theta_n}=X_0^{(n+1)}$ conditionally to 
 $\mu_{\theta_1},\ldots,\mu_{\theta_n}$ 
 and $Y_{\theta_0},Y_{\theta_1},\ldots,Y_{\theta_{n-1}}$ 
 is
  $\mu_{\theta_n}$. Moreover, $\{n<\upsilon_k\}$ is measurable with
  respect to $\mu_{\theta_1},\ldots,\mu_{\theta_n}$, and hence, for all
  $n\geq k$, (we denote by $\P^n$ the probability conditionally to
 $\mu_{\theta_1},\ldots,\mu_{\theta_n}$
  and $Y_{\theta_1},\ldots,Y_{\theta_{n-1}}$)
% }
  \begin{align*}
    \P^n(d(Z^k_n,\d D)\geq \eta)&= \P^n(d(Y_{\theta_n},\d D)\geq \eta)\11_{n<\upsilon_k}+\11_{n\leq\upsilon_k}\\
                              &\geq \mu_{\theta_n}(\{x\in D:d(x,\d D)\geq \eta\})\11_{n<\upsilon_k}+\11_{n\leq\upsilon_k}\geq 1-2\varepsilon.
  \end{align*}
  Using the law of large numbers for submartingales, this implies that, almost surely and for all $k\geq 1$,
  \[
    \liminf_{n\rightarrow+\infty} \frac{1}{n}\sum_{i=1}^n \11_{d(Z^k_i,\d D)\geq \eta}\geq 1-2\varepsilon.
  \]
  Observing that, almost surely, there exists $k\geq 1$ such that
  $Z^k_n=Y_{\theta_n}$ for all $n\geq 1$, this concludes the proof of~\eqref{eq:prop-tight-2a}.

  To prove~\eqref{eq:prop-tight-2b}, we observe that, due to the coupling argument of the proof of Proposition~\ref{prop:tightness1},
  \[
  \theta_n\geq\sum_{i=1}^n \11_{d(Z^k_i,\d D)\geq \eta}\bar{T}_0^{(i)},
  \]
  where $(\bar{T}_0^{(i)})_{i\geq 1}$ are i.i.d.\ copies of $\bar{T}_0$ such that $\bar{T}_0^{(i)}$ is independent of
  $Z^k_1,\ldots,Z^k_i$ for all $i\geq 1$. Therefore, we can use the law of large numbers for submartingales as above to conclude the
  proof of Proposition~\ref{prop:tightness2}.
\end{proof}

\subsection{Study of the empirical measure of the resampling points}
\label{sec:eta-n}

In this subsection, we focus on the behaviour of the random sequence
of measures $(\eta_n)_{n\geq 1}$ defined in~\eqref{def:eta-n}.
%  by
% \[
%   \eta_n=\frac{1}{n}\sum_{i=1}^n \delta_{Z_i},\text{ where }Z_i:=Y_{\theta_i}.
% \]
Our aim is to prove the following proposition using the theory of
pseudo-asymptotic trajectories.

\begin{prop}
  \label{prop:propSamplingPoints}
  The sequence of probability measures $(\eta_n)_{n\in\N}$ converges
  almost surely to $\alpha$ with respect to the weak topology.
\end{prop}

\begin{proof}
We follow an approach inspired from~\cite{Benaim1999}. Let
$(\tau_n)_{n\geq 1}$ be defined as $\tau_1=0$ and
$\tau_n=\gamma_2+\gamma_3+\cdots+\gamma_n$ for $n\geq 2$, where
\[
 \gamma_{n+1}=\frac{1}{(n+1)\eta_{n} A\11_D},\quad\forall n\geq 1.
\]
We consider the linearly interpolated version $(\widetilde\eta_t)_{t\in[1,+\infty)}$ of $(\eta_n)_{n\in\N}$ defined,
for all $n\geq 0$ and all $t\in \left[\tau_n,\tau_{n+1}\right]$, by
\begin{align*}
    \widetilde\eta_t=\eta_n+\frac{t-\tau_n}{\tau_{n+1}-\tau_n}(\eta_{n+1}-\eta_n),
\end{align*}
where we define by convention $\eta_0=\delta_{x_0}$ for some fixed $x_0\in D$.

Let $(f_k)_{k\in\N}$ be a sequence of bounded continuous functions
from $D$ to $\R$ such that the metric
\[
  d(\nu_1,\nu_2)=\sum_{k=0}^{\infty} \frac{|\nu_1 f_k -\nu_2
    f_k|
  }{2^k \|f_k\|_\infty},
\]
metrizes the weak topology on measures on $D$.

The main point of the proof consists in using \cite[Theorem~3.2]{Benaim1999} to prove that $\widetilde\eta$ is an asymptotic
pseudo-trajectory of~\eqref{eq:eqTheFlow}. By Proposition~\ref{prop:propTheFlow}, this means in our setting that, for all $T>0$,
  \begin{align}
    \label{eq:def-APT}
    \lim_{t\rightarrow +\infty} \sup_{s\in[0,T]} d\left(\widetilde\eta_{t+s},\,\frac{\widetilde\eta_t e^{sA}}{\widetilde\eta_t e^{sA}\11_D}\right)=0.
  \end{align}
This is stated in the next lemma, proved in  the next subsection.

\begin{lem}
  \label{lem:APT}
  The measure-valued process $\widetilde\eta$ is almost surely an asymptotic pseudo-trajectory for the distance $d$ on the set of
  probability measures on $D$ of the semi-flow induced by~\eqref{eq:eqTheFlow} and defined in Proposition~\ref{prop:propTheFlow}.
\end{lem}

Once this is proved, Proposition~\ref{prop:propSamplingPoints} follows
easily: indeed $(\widetilde\eta_t)_{t\geq 0}$ is almost surely a
relatively compact asymptotic pseudo-trajectory of the semi-flow
induced by~\eqref{eq:eqTheFlow} for which $\{\alpha\}$ is a compact
global attractor, which implies the result (see for
instance~\cite[Corollary~5.3]{BenaimCloezEtAl2017} and
\cite{Benaim1999,BenaiemHirsch1996}).

% Once this is proved, Proposition~\ref{prop:propSamplingPoints} follows easily: using Proposition~\ref{prop:tightness2}, almost
% surely, there exists $c>0$ such that $(\widetilde\eta_t)_{t\geq 0}$ belongs to the set $\{\mu\in {\cal P}(D),\ \mu(\{d(x,\partial D)\geq c\})\geq 1/2\}$. Now, it follows from Equation~\eqref{eq:prop-A-3} of
% Proposition~\ref{prop:prop-A} that, almost surely and for all $t,T\geq 0$,
% \[
%   d\left(\frac{\widetilde\eta_t e^{TA}}{\widetilde\eta_t e^{TA}\11_D},\,\alpha\right)\leq
%   \frac{2B}{\widetilde\eta_t(\eta)}
%   e^{-T\frac{\gamma}{\lambda_0(\lambda_0+\gamma)}}\leq \frac{4B}{c'}
%   e^{-T\frac{\gamma}{\lambda_0(\lambda_0+\gamma)}},
% \]
% where $c'=\inf_{x:d(x,\partial D)\geq c}\eta(x)>0$ since $\eta=\lambda_0\,A\eta$ is continuous by Proposition~\ref{prop:propLip}.
% Hence, it follows from~\eqref{eq:def-APT} that, for all $T\geq 0$,
% \[
% \limsup_{t\rightarrow+\infty} d(\widetilde\eta_{t},\alpha)=  \limsup_{t\rightarrow+\infty} d(\widetilde\eta_{t+T},\alpha)\leq \frac{4B}{c'}
% e^{-T\frac{\gamma}{\lambda_0(\lambda_0+\gamma)}}
% \]
% so that $\widetilde\eta_{t}$ converges to $\alpha$ when $t\rightarrow+\infty$ (note that the same holds true for any assymptotic
% pseudo trajectory evolving in $\{\mu\in {\cal P}(D),\ \mu\phi_D\geq c\}$, where we recall that $\phi_D$ is the distance to $\partial
% D$). This concludes the proof of Proposition~\ref{prop:propSamplingPoints}.

\end{proof}

\subsection{Proof of Lemma~\ref{lem:APT}}
\label{sec:pf-lemma-APT}

For all $n\geq 1$, we have
\[\eta_{n+1}-\eta_n = \frac{\delta_{Z_{n+1}}-\eta_n}{n+1} = \gamma_{n+1}\Big(F(\eta_n) + U_{n+1}\Big),\]
where, recalling the definition of $A$ in~\eqref{eq:A} and of $F$ in~\eqref{eq:def-F},
\[
  \gamma_{n+1}=\frac{1}{(n+1)\eta_n A\11_D}\quad\text{and}\quad U_{n+1}=(\eta_n A\11_D)\delta_{Z_{n+1}}-\eta_n A.
\]

Fix $\varepsilon\in(0,1/4)$ and $\eta>0$ small enough so that the conclusions of Proposition~\ref{prop:tightness2} hold true. Setting
$c:=\inf_{x\in D,\ d(x,\d D)>\eta} \frac{\delta_x A\11_D}{2}\wedge \frac{c_0\mathbb{E}\bar{T}_0}{2}$, which is positive by
Proposition~\ref{prop:propLip}, we define for all $k\geq 1$ the random variable in $\mathbb{N}\cup\{+\infty\}$
\[
  \sigma_k=\inf\{n\geq k,\ \eta_n A\11_D\leq c\text{ or }\theta_n\leq cn\}.
\]
The conclusion of Proposition~\ref{prop:tightness2} entails that
$\P(\cup_{k=1}^\infty \{\sigma_k=+\infty\})=1$.

Following~\cite{Benaim1999}, before proving Lemma~\ref{lem:APT}, we begin by proving the next lemma.
\begin{lem}
  \label{lem:martingale}
  Almost surely, for all bounded measurable function
  $f:D\rightarrow\R$, the numeric sequence 
  $\left(\sum_{\ell=1}^{n} \gamma_{\ell} U_{\ell}f\right)_n$ admits a
  finite limit  when $n\rightarrow+\infty$.
\end{lem}

\begin{proof}
  % {\color{blue}
  For all $\ell \geq 0$, we introduce $\mathcal{G}_\ell$ the $\sigma$-field generated by %\footnote{\color{blue} On allait jusqu'à $\ell+1$, c'était une typo?} 
  $\mu_{\theta_1},\ldots,\mu_{\theta_{\ell+1}}$, $\theta_1,\ldots,\theta_{\ell+1}$ and
  $Z_1,\ldots,Z_{\ell}$. %}
  Fix $k\geq 1$. We start by observing that
  \[
  \{\ell\leq\sigma_k\}=\Big\{\forall n\in\{k,k+1,\ldots,\ell-1\}, \eta_n A\11_D>c\text{ and }\theta_n>c n\Big\}\in\mathcal{G}_{\ell-1},
  \]
  so that $\sigma_k$ is predictable with respect to the filtration $(\mathcal{G}_\ell)_{\ell\geq 0}$.

  Following~\cite[Lemma~1]{Renlund2010}, we define $N_\ell = \gamma_\ell U_\ell f$ and
  \[
  M^{(k)}_n = \sum_{\ell=1}^{n\wedge\sigma_k} (N_\ell -
  \E_{\ell-1}N_\ell),
  \]
  where $\E_{\ell-1}$ denotes the expectation conditionally to $\mathcal{G}_{\ell-1}$. Observe $M^{(k)}_n$ is a martingale with
  respect to $(\mathcal{G}_\ell)_{\ell\geq 0}$ and that
  \[
  N_\ell=\frac{1}{\ell}\left(f(Z_\ell)-\frac{\eta_{\ell-1}Af}{\eta_{\ell-1}A\11_D}\right)\quad\text{and}\quad
  \E_{\ell-1}N_\ell=\frac{1}{\ell}\left(\mu_{\theta_\ell}f-\frac{\eta_{\ell-1}Af}{\eta_{\ell-1}A\11_D}\right).
  \]
  We have, for all $n\geq 0$
  \begin{align*}
    \mathbb E |M^{(k)}_n|^2
    &=  \sum_{\ell=1}^n \mathbb E\left[\big|N_\ell - \E_{\ell-1}N_\ell\big|^2\11_{\ell\leq \sigma_k}\right]\\
    &\leq 2\sum_{\ell=1}^n
    \E\left[|N_\ell|^2+|\E_{\ell-1} N_\ell|^2\right]\leq  4\sum_{\ell=1}^n \frac{\|f\|_\infty^2}{\ell^2}.
  \end{align*}
  As a consequence, the martingale $(M^{(k)}_n)_{n\geq 0}$ is uniformly
  bounded in $L^2$ and hence converges almost surely. Let us now prove 
  % In order to
  % prove Lemma~\ref{lem:martingale}, it is thus sufficient to prove
  that $\sum_{\ell=1}^{n\wedge\sigma_k}\E_{\ell-1}N_\ell$ converges
  almost surely when $n\rightarrow+\infty$.

  We have, for all $\ell\geq 1$,
  \begin{align*}
    \E\big|\E_{\ell-1}[N_\ell]\11_{\ell\leq\sigma_k}\big|
    % &= \E\left|\E_{\ell-1}\left[
    %   \frac{\eta_{\ell-1}A\11_D f(Z_\ell)-\eta_{\ell-1}A f}
    %   {\ell \eta_{\ell-1}A\11_D}\right]\11_{\ell\leq\sigma_k}\right| \\
    &=\frac{1}{\ell}\,\E\left|\mu_{\theta_\ell}f\,\11_{\ell\leq\sigma_k}-\frac{\eta_{\ell-1}Af}{\eta_{\ell-1}A\11_D}\11_{\ell\leq\sigma_k}\right|.
  \end{align*}
  For all $\ell< k$, this quantity is almost surely bounded by
  $2\|f\|_\infty/\ell$. For all $\ell\geq k$, the definition of $\sigma_k$ entails that
  \begin{align}
    \E\big|\E_{\ell-1}[N_\ell]\11_{\ell\leq\sigma_k}\big|
    \leq & \frac{1}{\ell}\,\E\left|\left(\frac{1}{\theta_\ell}-\frac{1}{(\ell-1)\eta_{\ell-1}A\11_D}\right)\theta_\ell\mu_{\theta_\ell}f
      \11_{\ell\leq\sigma_k}\right|\label{eq:substep2}\\     & +
      \frac{1}{c\ell(\ell-1)}\,\E\Big[\left|\theta_\ell \mu_{\theta_\ell}f - (\ell-1)\eta_{\ell-1}Af\right|\11_{\ell\leq\sigma_k}\Big]\label{eq:substep1}.
  \end{align}
  We first consider the term in~\eqref{eq:substep1}. It follows from the fact that $(\theta_{\ell+1}
  \mu_{\theta_{\ell+1}}f - \ell\eta_{\ell}Af)_{\ell\geq 0}$ is a $(\mathcal{G}_\ell)_{\ell\geq 0}$-martingale and from Cauchy-Schwarz
  inequality that
  \begin{align}
    \E\Big[|\theta_\ell \mu_{\theta_\ell}f - (\ell- & 1)\eta_{\ell-1}Af|\11_{\ell\leq\sigma_k}\Big]^2\nonumber\\
    &\leq \E\Big[\left|\theta_{\ell} \mu_{\theta_{\ell}}f - (\ell-1)\eta_{\ell-1}Af\right|^2\Big]\nonumber\\
    &= \sum_{i=1}^{\ell} \E\Big[\left|\int_0^{\tau_\d^{(i)}}f(X^{(i)}_s)ds - \delta_{Z_{i-1}}Af\right|^2\Big]\nonumber\\
    &\leq 2\|f\|_\infty^2 \sum_{i=1}^{\ell} \E((\tau^{(i)}_\d)^2)\leq 2 \|f\|_\infty^2 \ell \sup_{x\in D}\E_x(\tau_\d^2)\label{eq:substep1forsubstep2},
  \end{align}
  where $\sup_{x\in D} \E_x(\tau_\d^2)<+\infty$ since,   by~\eqref{eq:eta-borne-unif},
  \[
  \E_x(\tau_\d^2)=2\E_x\left(\int_0^{+\infty}t\, \11_{t<\tau_\d}\,dt\right)\leq 2\int_0^{+\infty}t\,\P_x(t<\tau_\d)\,dt\leq 2\int_0^{+\infty}t\,C''e^{-\lambda_0 t}\,dt.
  \]
  Consider now the term in~\eqref{eq:substep2}.
 \begin{multline*}
   \E\left|\left(\frac{1}{\theta_\ell}-\frac{1}{(\ell-1)\eta_{\ell-1}A\11_D}\right)\theta_\ell\mu_{\theta_\ell}f \11_{\ell\leq\sigma_k}\right|^2\\
   \begin{aligned}
   &\leq \|f\|_\infty^2\,\E\left|\left(1-\frac{\theta_\ell}{(\ell-1)\eta_{\ell-1}A\11_D}\right)^2 \11_{\ell\leq\sigma_k}\right|\\
   &\leq \frac{\|f\|_\infty^2}{c^2(\ell-1)^2}\,\E\left|\left((\ell-1)\eta_{\ell-1}A\11_D-\theta_\ell\right)^2 \11_{\ell\leq\sigma_k}\right|\\
   &\leq \frac{\|f\|_\infty^2}{c^2(\ell-1)^2}\,2 \ell \sup_{x\in D}\E_x(\tau_\d^2),
   \end{aligned}
 \end{multline*}
 where we used~\eqref{eq:substep1forsubstep2} with $f=\11_D$ to obtain the last inequality.

 We deduce that $\E|\E_{\ell-1}(N_\ell\11_{\ell\leq \sigma_k})|$ is ${\cal O}(\ell^{-3/2})$ (beware that the $\cal O$ may depend on
 $k$), so that
 $\E|\sum_{\ell=1}^{n\wedge\sigma_k}\E_{\ell-1}N_\ell|<+\infty$ and
 hence that $\sum_{\ell=1}^{n\wedge\sigma_k}\E_{\ell-1}N_\ell<\infty$ almost surely.

 Because of the a.s. convergence of the sequence $(M^{(k)}_n)_{n\in\N}$, we conclude
 that $(\sum_{\ell=1}^{n\wedge\sigma_k} N_\ell)_{n\in\N}$ converges
 almost surely when $n\rightarrow+\infty$ for all $k\geq 1$. Since, almost surely, there
 exists $k\geq 1$ such that $\sigma_k=+\infty$, this concludes the
 proof of Lemma~\ref{lem:martingale}.
\end{proof}

\begin{proof}[Proof of Lemma~\ref{lem:APT}]
We introduce the time-changed version $(\bar\eta_t)_{t\in[1,+\infty)}$ of the sequence  $(\eta_n)_{n\in\N}$ as $\bar\eta_t=\eta_n$ for all $n\geq 1$ and all
$t\in \left[\tau_n,\tau_{n+1}\right)$.
% \begin{align*}
%     \bar\eta_t =\eta_n\quad\text{and}\quad \tilde\eta_t=\eta_n+\frac{t-\tau_n}{\tau_{n+1}-\tau_n}(\eta_{n+1}-\eta_n).
% \end{align*}
We also define $\bar{U}_t=U_{n+1}$ for all
$t\in\left[\tau_n,\tau_{n+1}\right)$.

To apply \cite[Theorem~3.2]{Benaim1999}, one needs to prove that
$(\widetilde \eta_t)_{t\geq 0}$ is
almost surely
relatively compact, that it is
almost surely
uniformly continuous and that all limit points of
$(\Theta_t(\widetilde\eta))_{t\geq 0}$ in $C(\R_+,\cM(D))$, endowed
with the topology of uniform convergence for the metric $d$ on compact
time inervals, are
almost surely
weak solutions
of~\eqref{eq:eqTheFlow}, where
$\Theta_t(\widetilde\eta):=(\widetilde\eta_{t+s})_{s\geq 0}$.

The fact that $(\widetilde \eta_t)_{t\geq 0}$ is relatively compact is an immediate consequence of Proposition~\ref{prop:tightness2} and
the
almost surely
uniform continuity is also immediately obtained from the construction of $\widetilde\eta$, since for all
$s,t\in[\tau_n,\tau_{n+1}]$,
\begin{align}
  d(\widetilde{\eta}_s,\widetilde{\eta}_t) &
  \leq\sum_{k=0}^\infty\frac{|s-t|}{2^k\,\gamma_{n+1}\,\|f_k\|_\infty}\left|\frac{f_k(Z_{n+1})}{n+1}-\frac{f_k(Z_1)+\ldots+f_k(Z_n)}{n(n+1)}\right|
  \notag \\ &
  \leq \frac{4}{(n+1)\gamma_{n+1}}\,|s-t|
  \label{eq:unif-cont}
\end{align}
and since $\inf_{n\geq 1}(n+1)\gamma_{n+1}>0$
almost surely by Proposition~\ref{prop:tightness2}. 

In order to prove the last point, we adapt the method developed
in~\cite[Proposition~4.1]{Benaim1999}. Assume that there exists an
increasing sequence of positive numbers $(t_n)_{n\geq 0}$ converging
to $+\infty$ such that $(\Theta_{t_n}(\widetilde\eta))_{n\geq 0}$
converges to an element ${\widetilde\eta}^\infty$ in $C(\R_+,\cM(D))$ with
respect to the uniform convergence on compact time intervals. Our aim is to
prove that ${\widetilde\eta}^\infty$ is a weak solution to~\eqref{eq:eqTheFlow}.

For all $f\in C_b(D,\R_+)$, define
$L_F^f:C(\R_+,\cM(D))\rightarrow \R^{[0,+\infty)}$ by
\begin{align*}
  L_F^f(\nu)(t)=\nu_0 f+\int_0^t F(\nu_s) f\,ds,\quad\forall \nu\in C(\R_+,\cM(D)),
\end{align*}
so that, using the equality
$\int_t^{t+s} (F(\bar\eta_u)+\bar
U_u)\,du=-\widetilde\eta_t+\widetilde\eta_{t+s}$,
\begin{equation}\label{eq:Theta_t}
 \Theta_t(\widetilde\eta)f=L_F^f\big(\Theta_t(\widetilde\eta)\big)+A^f_t+B^f_t,
\end{equation}
where, for all $s\geq 0$,
\begin{align*}
A^f_t(s)=\int_t^{t+s} (F(\bar\eta_u)f-F(\widetilde\eta_u)f)\,du\quad\text{and}\quad B^f_t(s)=\int_t^{t+s} \bar{U}_u f\, du.
\end{align*}

For all $u\in[0,+\infty)$, let us denote by $n_u$ the unique
non-negative integer such that $u\in[\tau_{n_u},\tau_{n_u+1})$. Then, proceeding as in~\eqref{eq:unif-cont},
one easily checks that
\[
|\bar\eta_u g-\widetilde\eta_u g|\leq \frac{2\|g\|_\infty}{n_u+1},\quad \forall g\in C_b(D,\R_+).
\]
Since $n_u\rightarrow+\infty$ when $u\rightarrow+\infty$ and since
$F(\nu)f=\nu Af-(\nu A\11_D)\,\nu f$, where $Af$ and $A\11_D$ are
bounded continuous functions, we deduce
that $A^f_t(s)$ converges to $0$ when
$t\to \infty$.

Also, for all $t\in[\tau_n,\tau_{n+1})$ and $t+s\in [\tau_{n+m},\tau_{n+m+1})$, 
\begin{align*}
  |B^f_t(s)|&\leq (\tau_{n+1}-t) |U_{n+1} f| + \left|\sum_{\ell=n+1}^{n+m-1} \gamma_{\ell+1} U_{\ell+1}  f\right| + (s-\tau_{n+m}) \big|U_{n+m+1} f\big|\\
  &\leq \gamma_{n+1} \big|U_{n+1} f\big| + \left|\sum_{\ell=n+1}^{n+m-1} \gamma_{\ell+1} U_{\ell+1}  f\right| + \gamma_{n+m+1} \big|U_{n+m+1} f\big|.
\end{align*}
Hence Lemma~\ref{lem:martingale} implies that $B^f_t(s)$ also goes to $0$ when $t\rightarrow+\infty$.

Finally, since $L_F^f$ is clearly sequentially continuous in
$C([0,+\infty),\cM(D))$, one finally deduces that, for all
$f\in C_b(D,\R_+)$,
\[
  {\widetilde \eta}^\infty_t f={\widetilde\eta}^\infty_0 f+\int_0^t F({\widetilde\eta}^\infty_s)f\,ds,\ \forall t\geq 0,
\]
which means that ${\widetilde \eta}^\infty$ is a weak solution
to~\eqref{eq:eqTheFlow} and hence, by \cite[Theorem~3.2]{Benaim1999},
that $\widetilde\eta$ is an asymptotic pseudo-trajectory of the flow
induced by~\eqref{eq:eqTheFlow}.
\end{proof}

\subsection{Proof of Theorem~\ref{thm:main-d>1}}
\label{sec:end-pf}

% In view of Proposition~\ref{prop:propSamplingPoints}, to complete the proof of Theorem~\ref{thm:main-d>1}, it only remains to prove
% the almost sure convergence of $\mu_t f$ to $\alpha f$ for bounded measurable functions $f$, and the prove that
% $\theta_n/n\rightarrow 1/\lambda_0$ almost surely.

Fix  any bounded measurable functions $f:D\rightarrow\R$.  For all $n\geq 1$, we set
\[
\Psi_n=\theta_{(n+1)}\mu_{\theta_{(n+1)}}  f -n\eta_{n} A f.
\]
The random sequence $(\Psi_n)_{n\geq 1}$ is a
$(\mathcal{G}_\ell)_{\ell\geq 0}$-martingale and
  \[
    \Psi_{n} = \sum_{i=1}^{n} \int_{0}^{\tau_\d^{(i+1)}} f(X^{(i+1)}_s)ds - \delta_{Z_i}A f.
  \]
  This martingale property implies that
  \begin{align*}
    \frac{\E\left[|\Psi_n|^2\right]}{n}
    &\leq \frac{1}{n} \sum_{i=1}^n \mathbb E\left[\left| \int_{0}^{\tau_\d^{(i+1)}} f(X^{(i+1)}_s)ds - \delta_{Z_i}A f\right|^2\right]\\
    &\leq  2\|f\|^2_\infty\sup_{x\in D}\E_x(\tau_\d^2).
  \end{align*}
  From \cite[Theorem~1.3.17]{Duflo1997}, we deduce that $n^{-1}\Psi_n$
  goes almost surely to zero when $n$ goes to infinity,
  % .  Since this is
  % true for all $k\geq 1$ and since, almost surely, there exists
  % $k\geq 1$ such that $\sigma_k=+\infty$, we deduce that, almost
  % surely,
  that is
  \[
  \frac{\theta_{(n+1)}\mu_{\theta_{n+1}}  f}{n} -\eta_{n} A f\xrightarrow[n\rightarrow+\infty]{a.s.} 0.
  \]
 Since $Af$ is continuous and bounded for any bounded measurable
 function $f$ (see Proposition~\ref{prop:propLip}),
 one deduces from Proposition~\ref{prop:propSamplingPoints} that, almost surely,
 \[
   \frac{\theta_{n}\mu_{\theta_n}  f}{n} \xrightarrow[n\rightarrow+\infty]{} \alpha Af =\alpha f/\lambda_0.
 \]
 Applying this result to $f=\11_D$, one deduces that $\theta_n/n$
 converges to $1/\lambda_0$ almost surely and hence that
 $\mu_{\theta_n}f$ converges to $\alpha f$ almost surely. Since, for all $t\in[\theta_n,\theta_{n+1})$,
\begin{align*}
\left|\mu_tf-\mu_{\theta_n}f\right|&\leq
\frac{1}{t}\left|\int_{\theta_n}^t f(Y_s)\,ds\right|+\left|\int_0^{\theta_n} f(Y_s)\,ds\right|\left(\frac{1}{t}-\frac{1}{\theta_n}\right)\\
&\leq 
\frac{\|f\|_\infty}{t}(t-\theta_n)+\frac{(t-\theta_n)}{t\,\theta_n}\,\theta_n\,\|f\|_\infty\\
&\leq 2\|f\|_\infty\left(1-\frac{\theta_n}{\theta_{n+1}}\right),
\end{align*}
the almost sure convergence of $\mu_t$ to $\alpha f$ when $t\rightarrow+\infty$ follows from the almost sure convergence of
$\theta_n/n$ to the positive constant $1/\lambda_0$.

\bibliographystyle{abbrv} \bibliography{biblio}

\end{document}